\documentclass[10pt,letterpaper]{article}

\usepackage[utf8]{inputenc}
\usepackage[T1]{fontenc}

\usepackage{url,graphicx,tabularx,array,geometry,tikz,soul,mathtools}
\usepackage{amsmath,bm}
\usepackage{ amssymb }
\usepackage{mathrsfs}
\usepackage{amsthm,stmaryrd}
\usepackage[all,cmtip]{xy}

\usepackage{cite}

\usepackage{changepage}

\usepackage[aboveskip=1pt,labelfont=bf,labelsep=period,singlelinecheck=off]{caption}

\makeatletter
\renewcommand{\@biblabel}[1]{\quad#1.}
\makeatother

\usepackage{lastpage,fancyhdr,graphicx}
\usepackage{epstopdf}
\usepackage{commands}

\reversemarginpar

\begin{document}
\vspace*{0.35in}

\begin{flushleft}
{\Large
\textbf\newline{A Comparison of Overconvergent Witt de-Rham Cohomology and Rigid Cohomology on Smooth Schemes}
}
\newline
\\
Nathan Lawless

\end{flushleft}

\section*{Abstract}
We generalize the functorial quasi-isomorphism in \cite{Davis2011} from overconvergent Witt de-Rham cohomology to rigid cohomology on smooth varieties over a finite field $k$, dropping the quasi-projectiveness condition. We do so by constructing an \et ale hypercover for any smooth scheme $X$, refined at each level to be a disjoint union of open standard smooth subschemes of $X$. We then find, for large $N$, an $N$-truncated closed embedding into a simplicial smooth scheme over $W(k)$, which allows us to use the results of \textit{loc. cit} at the simplicial level, and use cohomological descent to prove the comparison.

\section{Introduction}

Let $X$ be a smooth scheme over a perfect field $k$ of characterstic $p>0$, and consider its overconvergent de Rham-Witt complex of Zariski sheaves $W^\dagger\Omega^\bullet_{X/k}$, which is  defined in \cite{Davis2011} (see Definition 1.1 and Theorem 1.8). One of the main results of \textit{loc. cit} is that if $X$ is also quasi-projective, then there exists a natural quasi-isomorphism
$$\Rig(X/K)\stackrel{\sim}{\to} R\Gamma(X,W^\dagger\Omega^\bullet_{X/k})_\Q,$$
where $K=W(k)\otimes \Q$. 

The main result of this paper is Theorem \ref{comparison theorem}, where we drop the quasi-projectivity condition in the comparison. We outline the approach in \cite{Davis2011} and the one used in our paper.

If $X=\Spec A$, \cite{Davis2011} consider pairs $(X,F)$ given by closed immersions 
$$X=\Spec A\into F=\Spec \tilde{A}$$
into $W(k)$-schemes, called special frames. To this, the authors associate dagger spaces (in the sense of \cite{Grosse-Klonne2000}) $]X[_{\Fa}^\dagger$ functorially in $(X,F)$, which calculate $\Rig(X/K)$:
\begin{align}\label{dagger to rigid1}\Rig(X/K)\stackrel{\sim}{\to} R\Gamma(]X[_{\Fa}^\dagger,\Omega^\bullet_{]X[_{\Fa}^\dagger}).\end{align}
so using the specialization maps
$$sp_*: ]X[_{\Fa}^\dagger\to X$$
we have that $\text{R}\Gamma_{rig}(X/K)\cong R\Gamma(X,\text{Rsp}_*\Omega^\bullet_{]X[_{\Fa}^\dagger})$.

They also form a quasi-isomorphism of Zariski sheaves on $X$,
\begin{align}\label{dagger to oWdR1}\text{sp}_*\Omega^\bullet_{]X[_{\hat{F}}^\dagger}\to W^\dagger\Omega^\bullet_{X/k}\otimes \Q,\end{align}
functorial in $(X,F)$.

For the general case, one could try to work locally by introducing an affine covering of $X$, and let $X_0$ be its disjoint union, fitting into a special frame $(X_0,F_0)$, and then through a $0$-coskeleton work on some simplicial frame $(X_\bullet,F_\bullet)$. Unfortunately, by (\ref{dagger to rigid1}) and (\ref{dagger to oWdR1}) below this requires proving vanishing of the higher cohomologies of $\text{Rsp}_*\Omega^\bullet_{]X_n[_{\Fa_n}^\dagger}$, which is not known in general. In the case where $X$ is smooth and quasi-projective (though possibly not affine), one can take an open covering by a particular type of affine smooth schemes, standard smooth schemes, which may be lifted nicely over $W(k)$, which are all coming from localizations in a projective space. This gives a nice description of the intersections of such opens in $\cosk_0^X(X_0)_\bullet$, which allows them to prove the desired vanishing of higher cohomologies, and then complete the proof by means of cohomological descent.

For our case, when $X$ is not quasi-projective, we do not have a common projective space in which all our open affines are open. So instead of working with the 0-coskeleton, we refine it at each level, getting an \et ale hypercovering $X_\bullet/X$ so that each level $X_n$ is a disjoint union of affine standard smooth open subschemes of $X$, which we call a special hypercovering. This is done in Section \ref{hypercovering}.

In Section \ref{frame}, this hypercovering needs to be embedded into a simplicial smooth scheme over $W(k)$ in order to form a simplicial special frame and dagger space on which to apply (\ref{dagger to oWdR1}). We use Tsuzuki's functor $\Gamma_N^{W(k)}(-)$ introduced in \cite{Chiarellotto2003} to form an $N$-truncated simplicial special frame $(X_{\bullet\leq N},F_{\bullet\leq N})$, with $N$ large enough so that the $X_m$ for $m>N$ don't contribute to the calculation of $\t{R}\Gamma_{rig}(X/K)$.

The comparison is then proven in \label{main_proof} in various steps:
\begin{itemize}
\item Prove the vanishing of $\t{Rsp}_*^i\Omega^\bullet_{]X_n[_{\Fa_n}^\dagger}$ for $0\leq n \leq N$ and $i>0$: this is done using techniques from the proof of \cite[Proposition 4.35]{Davis2011}, such as being able to replace the $F_n$ by some $F'_n$ \et ale over $F_n$ or equal to $F_n\times_{W(k)}\A^r_{W(k)}$ for some $r$ fitting into a special frame $(X_n,F'_n)$.
\item Prove that the complex $R\Gamma(X_{\bullet\leq N},\text{Rsp}_*\Omega^\bullet_{]X_{\bullet\leq N}[_{\Fa_{\bullet\leq N}}^\dagger})$ calculates $\t{R}\Gamma_{\t{rig}}(X/K)$ for large enough $N$: this is motivated by \cite{Nakkajima2012} and relies on the machinery of \cite{Chiarellotto2003}, such as vanishing of higher enough rigid cohomology groups of $X$, independence of the choices of rigid frames and cohomological descent methods.
\item Prove that the isomorphism $$\Rig(X/K)\cong \t{R}\Gamma(X_{\bullet\leq N},\text{Rsp}_*\Omega^\bullet_{]X_{\bullet\leq N}[_{\Fa_{\bullet\leq N}}^\dagger})$$ in $D_+(K)$ is independent of choices made, and functorial in $X$: this is done by "refining" any two choices made to a common one.
\end{itemize}

\section{Background}
\subsection{Rigid Cohomology}
We refer to \cite{Chiarellotto2003} for results and notations involving rigid cohomology. We just make the following changes of notation:
\begin{itemize}
\item Given a simplicial morphism of triples $w_\bullet=(\mathring{w}_\bullet,\overline{w}_\bullet,\hat{w}_\bullet):(Y_\bullet,\overline{Y}_\bullet,\mathcal{X}_\bullet)\to (X,\overline{X},\mathcal{X})$, and a complex of sheaves $E_\bullet^\bullet$ of coherent $\tilde{w}_\bullet\inv j^\dagger\O_{]\overline{X_\bullet}[_{\Xa_\bullet}}$-modules we set
$$\t{R}w_*E:=\t{R}\mathcal{C}^\dagger\left((X,\overline{X},\mathcal{X}),(Y_\bullet,\overline{Y}_\bullet,\mathcal{X}_\bullet); E_\bullet^\bullet\right),  $$
which is defined in Section 4.2.
We do this change of notation to match with simplicial notation in \cite{Conrad2003}, as both complexes are defined as the total complex associated to
$$\xymatrix{... & ... & ... & ...\\
0\ar[r] &\mathcal{I}_0^1\ar@<-.5ex>[r] \ar@<.5ex>[r]\ar[u] & \mathcal{I}_1^1\ar@<-.8ex>[r] \ar@<.8ex>[r]\ar[r]\ar[u]& ...\\
0\ar[r] & \mathcal{I}_0^0\ar@<-.5ex>[r] \ar@<.5ex>[r]\ar[u] & \mathcal{I}_1^0\ar@<-.8ex>[r] \ar@<.8ex>[r]\ar[r]\ar[u]& ...\\
 & 0\ar[u] & 0\ar[u] & ...}$$
where $\mathcal{I}_\bullet^\bullet$ is an injective resolution of $E_\bullet^\bullet$, where the vertical maps come from maps in $\mathcal{I}_p^\bullet$, and the horizontal come from the simplicial structure.

We note that for a $N$-truncated simplicial map, we will have a similar complex, but with all columns being 0 after $N$.
\item Given an open immersion of $X$ into a a proper scheme $\overline{X}$ over $k$, and a universally de Rham descendable hypercovering $(Y_\bullet,\overline{Y}_\bullet,\mathcal{Y}_\bullet)$ of $(X,\overline{X})$ (as in \cite[Definition 10.1.3]{Chiarellotto2003}) we will set
$$\Rig(X/K)=\t{R}\Gamma(]\overline{Y}_\bullet[_{\Ya_\bullet},j^\dagger\Omega^\bullet_{]\overline{Y}_\bullet[_{\Ya_\bullet}})$$
where the right hand side is seen as coming from the map into the triple $(\Spec k,\Spec k,\Spf \Wa)$. This definition coincides with theirs since we are dealing with trivial coefficients.
\end{itemize}

We also recall the following definitions (see Definition 7.2.1., 7.2.2. and 11.3.1 in \cite{Chiarellotto2003}):
\begin{defn}\begin{itemize}
\item[1)] For a simplicial scheme $Y_\bullet$, separated of finite type over a scheme $X$, we say that $Y_\bullet\to X$ is an \textit{\et ale} (resp. \textit{proper}) \textit{hypercovering} if the canonical morphism
$$Y_{n+1}\to \cosk_n^X(\sk_n^X(Y_\bullet))_{n+1}$$
is \et ale surjective (resp. proper surjective) for any $n$.
\item[2)] For a simplicial pair of schemes $(Y_\bullet,\overline{Y}_\bullet)$ of schemes separated of finite type over some pair $(X,\overline{X})$, we say that $(Y_\bullet,\overline{Y}_\bullet)\to (X,\overline{X})$ is an \textit{\et ale-proper hypercovering} if $Y_\bullet\to X$ is an \et ale hypercovering and $\overline{Y}_\bullet\to \overline{X}$ is a proper hypercovering.
\item[3)] For a simplicial triple of schemes $(Y_\bullet,\overline{Y}_\bullet,\Ya_\bullet)$ separated of finite type over some triple $(X,\overline{X},\Xa)$, we say that $(Y_\bullet,\overline{Y}_\bullet,\Ya_\bullet)\to (X,\overline{X},\Xa)$ is an \textit{\et ale-proper hypercovering} if  $(Y_\bullet,\overline{Y}_\bullet)\to (X,\overline{X})$ is one, and the natural maps
$$\cosk_{n+1}^\Xa(\sk_{n+1}^\Xa(\Ya_\bullet))_k\to \cosk_{n}^\Xa(\sk_{n}^\Xa(\Ya_\bullet))_k$$
are smooth around $\cosk_{n+1}^X(\sk_{n+1}^X(Y_\bullet))_k$ for any $n$ and $k$.
\item[4)] Given an \et ale hypercovering $Y_\bullet\to X$, we say that $V_\bullet$ is a refinement of $Y_\bullet$ if it fits into a diagram
$$\xymatrix{V_\bullet\ar[r]\ar[dr] & Y_\bullet\ar[d]\\
& X}$$
where $V_\bullet\to X$ is an \et ale hypercovering, and the induced morphisms
$$V_{n+1}\to \cosk_{n}^X(\sk_n^X(V_\bullet))_{n+1}\times_{\cosk_{n}^X(\sk_n^X(Y_\bullet))_{n+1}} Y_{n+1}$$
are \et ale surjective for each $n$.
\item[5)]Given an \et ale-proper hypercovering $(Y_\bullet,\overline{Y}_\bullet)\to (X,\overline{X})$, we say that $(V_\bullet,\overline{V}_\bullet)$ is a \textit{refinement of}  $(Y_\bullet,\overline{Y}_\bullet)$ if it fits into a diagram of pairs
$$\xymatrix{(V_\bullet,\overline{V}_\bullet)\ar[r]\ar[dr] & (Y_\bullet,\overline{Y}_\bullet)\ar[d]\\
& (X,\overline{X})}$$
where $V_\bullet$ is a refinement of $Y_\bullet$ over $X$.
\end{itemize}
\end{defn}

\subsection{Special Frames and Dagger Spaces}\label{special + dagger}

The following is a summary of Section 4 of \cite{Davis2011}.

\begin{defn} A \textit{special frame} is a pair $(X,F)$ with a closed embedding $X\into F$, where $X$ and $F$ are smooth affine schemes over $k$ and $W(k)$ respectively.\end{defn}

Given a special frame $(X,F)$, we can choose an embedding $F\into \A^n_{W(k)}$ for some $n$, and in turn we have an open embedding $E:=\A^n_W\into \P^n_{W(k)}=:P$.  Let $Q=\overline{F}$ and $\overline{X}$ be the closures of $F$ and $X$ respectively in $P$, and let $\Fa$ and $\Qa$ be the $p$-adic completions of $F$ and $Q$ respectively. Then, 
$$X\into \overline{X}\into \Qa$$
 is a frame for rigid cohomology in the sense of Berthelot (i.e. we have an open immersion of $X$ into a proper scheme $\overline{X}$ over $k$, and a closed immersion $\overline{X}\into\hat{Q}$ where $\hat{Q}$ is smooth around $X$). So we may define the rigid cohomology of $X$ as 
$$\Rig(X/K)=R\Gamma(]\overline{X}[_{\Qa},j^\dagger\Omega^\bullet_{]\overline{X}[_{\Qa}}),$$
where $j$ is the inclusion $]X[_{\Qa}\into ]\overline{X}[_{\Qa}$. Note also that $]X[_{\Qa}=]X[_{\Fa}$.

The authors then give an explicit description of a fundamental system of strict neighborhoods of $]X[_{\Fa}$ in $]\overline{X}[_{\Qa}$, which they use to give a dagger structure (in the sense of \cite{Grosse-Klonne2000}) on $]X[_{\Fa}$, denoted by $]X[_{\Fa}^\dagger$, along with a morphism
$$sp_*: ]X[_{\Fa}^\dagger\to X$$
which is independent of the choice of embedding of $F$ into affine and projective space. Thus, we have an association
$$(X,F)\mapsto ]X[_\Fa^\dagger$$
of special frames into dagger spaces, functorial in $(X,F)$. 

By \cite[Theorem 5.1]{Grosse-Klonne2000}, this gives quasi-isomorphisms
\begin{align}\label{dagger to rigid}\Rig(X/K)=R\Gamma(]\overline{X}[_{\Qa},j^\dagger\Omega^\bullet_{]\overline{X}[_{\Qa}})\stackrel{\sim}{\to} R\Gamma(]X[_{\Fa}^\dagger,\Omega^\bullet_{]X[_{\Fa}^\dagger}).\end{align}

To such a frame $(X,F)$, they also form in \cite[(4.32)]{Davis2011} a map
\begin{align}\label{dagger to oWdR}\text{sp}_*\Omega^\bullet_{]X[_{\hat{F}}^\dagger}\to W^\dagger\Omega^\bullet_{X/k}\otimes \Q,\end{align}
which is a quasi-isomorphism of Zariski sheaves. 

\subsection{Standard Smooth Schemes}
\begin{defn}We call a ring $A$ a \textit{standard smooth algebra } (over $k$) if $A$ can be represented in the form
$$A=k[X_1,...,X_n]/(f_1,...,f_m),$$
where $m\leq n$ and the determinant
$$\t{det}\left(\frac{\partial f_i}{\partial X_j}\right), \,\, 1\leq i,j\leq m$$
is a unit in $A$. The scheme $\Spec A$ is then called a \textit{standard smooth scheme}.\end{defn}

Such schemes are convenient to work with, since for a standard smooth algebra represented as $k[T_1,...,T_n]/(f_1,...,f_r)$, we may choose liftings $\tilde{f}_1,...,\tilde{f}_r$ to $W[T_1,...,T_n]$, and let $\tilde{A}$ be the localization of $W[T_1,...,T_n]/(\tilde{f}_1,...,\tilde{f}_r)$ with respect to $\t{det}\left(\frac{\partial \tilde{f}_i}{\partial T_j}\right)$. Then, $\tilde{A}$ is a standard smooth algebra which lifts $A$ over $W$, which gives a special frame $(\Spec A, \Spec \tilde{A})$. We note that this may be done functorially in $A$; that is, given a homomorphism of standard smooth algebras
$$\varphi: A\to B$$
with presentations 
$$A\cong k[T_1,..,T_n]/(f_1,..,f_r), \, B\cong k[S_1,...,S_m]/(g_1,...,g_s),$$
after choosing liftings $\tilde{f}_i$ to define $\tilde{A}$, we may chose the representation 
$$B\cong k[S_1,...,S_m,T_1,...,T_n]/(g_1,...,g_s,f_1,...,f_r,T_1-\alpha(T_1),...,T_r-\alpha(T_r))$$
and then take liftings $\tilde{g_j}, \tilde{\alpha_i}$ over $g_j$ and $\alpha(T_i)$ respectively to form $\tilde{B}$. 

Note also that for any such standard smooth scheme $F=\Spec \tilde{A}$, we have an \et ale map 
$$F\to \A^n_{W(k)}$$
for some $n$.

\section{The hypercovering}\label{hypercovering}

We recall the definition of a split simplicial scheme from \cite[Definition 4.9]{Conrad2003}
\begin{defn} We say that a simplicial scheme $Y_\bullet$ is \textit{split} if there exist subobjects $NY_j$ in each $Y_j$ such that the natural map
$$\bigsqcup\limits_{\phi: [n]\surj [m]}NY_\phi\to Y_n$$
is an isomorphism for every $n\geq 0$, where $NY_\phi:=NY_m$ for a surjection $\phi:[n]\surj [m]$, and the natural  maps are given by the composition
$$NY_\phi\subset Y_m\stackrel{Y_\bullet(\phi)}{\to} Y_n.$$

The truncated case is defined similarly.
\end{defn}

We denote by $NY_{m,\phi}$ the image of $NY_\phi\subset Y_m$ under this isomorphism. Notice that this agrees with the definition in \cite[Section 11.2]{Chiarellotto2003} as for any epimorphism $\phi:[n]\surj [m]$ we have a commutative map
$$\xymatrix{ &NY_{m,\t{id}_{[m]}}\ar@{}[r]|-*[@]{\subset}\ar[dd]^\sim & Y_m\ar[dd]^{Y_\bullet(\phi)}\\
NY_m\ar[ur]^{\sim}\ar[dr]^{\sim}\\
& NY_{n,\phi}\ar@{}[r]|-*[@]{\subset} & Y_n}.$$

Next, by \cite[Theorem 4.12]{Conrad2003}, given any split $n$-truncated simplicial scheme $Y_{\bullet\leq n}/X$ with the splitting given by $\{NY_k\}_{0\leq k\leq n}$, in order to extend it to a split $(n+1)$-truncated scheme $Y_{\bullet\leq n+1}/X$ it suffices to give a scheme $N$ and a morphism
$$\beta: N\to \cosk^X_n(Y_{\bullet\leq{n}})_{n+1}.$$

This coincides with the functor $\Omega^X_{n+1}(Y_{\bullet\leq n},NY_0,...,NY_{n+1})$ given in \cite[Section 11.2]{Chiarellotto2003}.

\begin{prop}\label{refinement} Given any \et ale hypercovering $Z_\bullet\to X$, with $Z_n$ being smooth schemes over $k$, there exists an \et ale hypercovering $Y_\bullet\to X$ refining $Z_\bullet\to X$ such that for any $n$, $Y_n$ is the disjoint union of affine standard smooth schemes giving a finite open covering of $Z_n$.
\end{prop}
\begin{proof}
The proof is nearly identical to \cite[Proposition 11.3.2]{Chiarellotto2003}, with the only difference being that when we form a finite affine Zariski covering of
$$\cosk_{n}^X(Y_{\bullet\leq n})_{n+1}\times_{\cosk_{n}^X(Z_{\bullet\leq n})_{n+1}} Z_{n+1},$$
we require the covering to be by affine standard smooth schemes also. 
\end{proof}

\begin{defn} We say $Y_\bullet\to X$ is a \textit{special hypercovering} if $Y_\bullet$ is a split \et ale hypercovering of $X$, and each $Y_n$ is a disjoint union of affine standard smooth schemes which give an open covering of $X$.\end{defn}

We prove the existence and some functorial property of such hypercoverings, which will be useful to work on the comparison locally.

\begin{prop}\label{sp hypercovering} Given a smooth scheme $X$: 
\begin{itemize}
\item[i)] There exists a special hypercovering $Y_\bullet \to X$.
\item[ii)] Given two special hypercoverings $Y_\bullet,Y'_\bullet$/X, there a third special hypercovering $Y''_\bullet/X$ refining them.
\item[iii)] Given a morphism $X\to X'$ of smooth schemes, there exist special hypercoverings $Y_\bullet\to X$ and $Y'_\bullet\to X'$ fitting in a commutative diagram
$$\xymatrix{Y_\bullet\ar[r]\ar[d] & Y'_\bullet\ar[d]\\
X\ar[r] & X'}.$$
\end{itemize}
\end{prop}
\begin{proof}
Part i) follows immediately from Proposition \ref{refinement} by taking the constant simplicial scheme $Z_\bullet=\cosk_{-1}^X(X)$ (so $Z_n=X$ for all $n$). For part ii), we just apply Proposition \ref{refinement} with
$$Z_\bullet:=Y_\bullet \times_X Y'_\bullet,$$
and for part iii) we find some special hypercovering $Y'_\bullet\to X'$, and then again use Proposition \ref{refinement} with
$$Z_\bullet:=Y'_\bullet\times_{X'}X.$$
\end{proof}

The following proposition will be used to find an open embedding of these special hypercoverings into proper hypercoverings of some compactification $\overline{X}$ of $X$ in a functorial manner:

\begin{prop}\label{simplicial compactification} Given a split \et ale hypercovering $X_\bullet\to X$, and an open embedding $X\into \overline{X}$ into a proper $k$-scheme:
\begin{itemize}
\item[i)] There exists a split \et ale-proper hypercovering $(X_\bullet,\overline{X}_\bullet)\to (X,\overline{X})$ such that for each $n$, $X_n\to \overline{X_n}$ is an open embedding.
\item[ii)] Given a morphism of compactifications $(X,\overline{X})\to (Y,\overline{Y})$, and a morphism of split \et ale-hypercoverings $X_\bullet\to Y_\bullet$ (over $X$ and $Y$) given by morphisms $NX_k\to NY_k$, we may form $\overline{X}_\bullet$ and $\overline{Y}_\bullet$ as in i), with a morphism fitting into the commutative diagram
$$\xymatrix{(X_\bullet,\overline{X}_\bullet)\ar[r]\ar[d] & (Y_\bullet,\overline{Y}_\bullet)\ar[d]\\
(X,\overline{X})\ar[r] & (Y,\overline{Y})}.$$
\end{itemize}
\end{prop}
\begin{proof}
Part i) is \cite[Proposition 11.7.3]{Chiarellotto2003}, but we still explain it. First note that given any compatification $V\into \overline{V}$ of $k$-schemes, and a morphism $U\to V$, by Nagata's compactification theorem for $U\to \overline{V}$ we may choose a compactification $U\into \overline{U}$ fitting into a diagram
$$\xymatrix{U\ar[r]\ar[d] & \overline{U}\ar[d]\\ V\ar[r] &\overline{V}}$$
where $\overline{U}$ is proper over $\overline{V}$ (and thus over $k$). In this case we will say that $\overline{U}$ is a compactification of $U$ over $(V,\overline{V})$.

Let $X_\bullet$ be giving by a splitting $\{NX_k\}$. We construct $\overline{X}_\bullet$ by a splitting at each level. First, we set $N\overline{X}_0=\overline{X}_0$ to be a compactification of $NX_0=X_0$ over $(X,\overline{X})$. 

Next, having constructed a split $n$-truncated $\overline{X}_{\bullet\leq n}$ with a splitting $\{N\overline{X}_k\}_{0\leq k\leq n}$ and open embeddings $NX_k\into N\overline{X}_k$ into proper $k$-schemes, we let $N\overline{X}_{n+1}$ to be the compactification of $NX_{n+1}$ over
$$\cosk_n^X(\sk_n^X(X_\bullet)))_{n+1}\into \cosk_n^{\overline{X}}(\overline{X}_{\bullet\leq n})_{n+1}.$$
The above is an open immersion since all $X_k\into \overline{X}_k$ and $X\into \overline{X}$ are, and similarly the right hand side is proper. Then, letting
$$\overline{X}_{\bullet\leq n+1}=\Omega(\overline{X}_{\bullet\leq n+1},N\overline{X}_0,...,N\overline{X}_{n+1})$$
we have an $n+1$-truncated proper hypercovering of $\overline{X}$, with an open immersion coming from $X_{\bullet\leq n+1}$.

For ii), we construct $\overline{Y}_\bullet$ as in i). Then, we build $\overline{Y}_\bullet$ similarly, except that at each $n$, we take a compactification $N\overline{X}_{n+1}$ of $NX_{n+1}$ over
$$\cosk_n^X(\sk_n(X_\bullet))_{n+1}\times_{\cosk_n^Y(\sk_n(Y_\bullet))_{n+1}}NY_{n+1}\into \cosk_n^{\overline{X}}(\sk_n(\overline{X}_\bullet))_{n+1}\times_{\cosk_n^{\overline{Y}}(\sk_n(\overline{Y}_\bullet))_{n+1}}N\overline{Y}_{n+1}.$$
This all fits into a commutative diagram
$$\xymatrix{NX_{n+1}\ar[r]\ar[d] & N\overline{X}_{n+1}\ar[d]\\
\cosk_n^X(\sk_n(X_\bullet))_{n+1}\times_{\cosk_n^Y(\sk_n(Y_\bullet))_{n+1}}NY_{n+1}\ar[r]\ar[d] & \cosk_n^{\overline{X}}(\sk_n(\overline{X}_\bullet))_{n+1}\times_{\cosk_n^{\overline{Y}}(\sk_n(\overline{Y}_\bullet))_{n+1}}N\overline{Y}_{n+1}\ar[d]\\
\cosk_n^X(\sk_n(X_\bullet))_{n+1}\ar[r] & \cosk_n^{\overline{X}}(\sk_n(\overline{X}_\bullet))_{n+1}}$$
where all horizontal morphisms are open immersions, and the vertical morphisms on the right are all proper. This gives us the desired functoriallity.
\end{proof}

\section{The simplicial special frame}\label{frame}

We explain the construction of the Tsuzuki functor $\Gamma^\Ca_N(-)$, introduced in \cite[Section 11.2]{Chiarellotto2003}. Given a category $\Ca$ with finite inverse limits, a non-negative integer $N$, and an object $Z$, we construct a $N$-truncated simplicial object $\Gamma^\Ca_N(Z)$ in $\text{Simp}_{\leq N}(\Ca)$ as follows:

Set $$\Gamma^\Ca_N(Z)_m:=\prod_{\phi:[N]\to[m]}Z_\phi$$
where $Z_\phi=Z$. To define the simplicial maps, given $\alpha: [m']\to[m]$, we define $\Gamma_\alpha: \Gamma^\Ca_N(Z)_m\to \Gamma^\Ca_N(Z)_{m'}$ by
$$(c_\phi)_{\phi: [N]\to[m]}\mapsto (d_\psi)_{\psi: [N]\to[m']}$$
where $d_\psi:= c_{\alpha\circ \psi}$.

Given any $Y_{\bullet\leq N}$ in $\text{Simp}_{\leq N}(\Ca)$, and a morphism $f: Y_N\to Z$ in $\Ca$, we construct a morphism
$$Y_{\bullet\leq N}\to \Gamma^\Ca_N(Z)_{\bullet\leq N}$$
by the commutative diagram
$$\xymatrix{Y_m\ar[r]\ar[d]_{Y(\phi)} & \Gamma^\Ca_N(Z)_m=\prod_{\phi:[N]\to[m]}X_\phi\ar[d]^{p_\phi}\\
Y_N\ar[r]^f & Z=Z_\phi}$$
for any $m$ and $\phi:[N]\to [m]$, where $p_\phi$ is just the projection on to the $\phi: [N]\to [m]$ factor.

Letting $\Ca$ be the category of schemes over $\t{Spec}(W(k))$, and $Y_{\bullet\leq N}$ some simplicial scheme over $k$ or $W(k)$, we have the following:
\begin{lem}\label{nak6.6.1}
If $f: Y_N\to W$ is a closed immersion over $W(k)$, and $Y_{\bullet\leq N}$ and $Z$ are separated, then the induced morphism
$$Y_{\bullet\leq N}\to \Gamma^{W(k)}_N(Z)_{\bullet\leq N}$$
is a closed immersion of $N$-truncated schemes.
\end{lem}
\begin{proof}
For any $0\leq m \leq {N}$, consider any face morphism $d:Y_{m}\to Y_{N}$ (with $d=id_{Y_N}$ if $m=N$), and a corresponding degeneracy map $s: Y_N\to Y_m$ which is a section to $d$. Then, we have
$$\xymatrix{Y_N\ar[r]\ar[dr] & Y_m\ar[d]\\
& W(k)}$$
where the vertical and diagonal maps are separated. This shows that $d$ is also separated. Then, by the commutative diagram
$$\xymatrix{Y_m\ar[r]^{s}\ar@{=}[dr] & Y_N\ar[d]^{d}\\
& Y_m}$$
we see that $s$ is a closed immersion.  Finally, by the definition of the map $g_m: Y_m\to \Gamma^{W(k)}_N(Z)_m$, we have a commutative diagram
$$\xymatrix{Y_m\ar[r]^(.35){g_m}\ar[d]^s & \Gamma^{W(k)}_N(Z)_m\ar@{=}[r]\ar[d]^{pr_{s}} &\prod\limits_{\phi: [N]\to [m]}Z\\
Y_N\ar[r]^{f} & Z}$$
which shows that $g_m$ is in fact a closed immersion.
\end{proof}

Thus, for a smooth scheme $X$ over $k$, having formed a special hypercovering $Y_\bullet\to X$ as in the previous section, for any $N\geq 0$, we may write 
$$Y_N=\bigsqcup\limits_{\phi: [N]\surj [m]} NY_{N,\phi}$$
where for $\phi:[N]\surj [m]$, $NY_{N,\phi}$ is identified with $NY_m$ under the map $Y_m\stackrel{Y_\bullet(\phi)}{\to} Y_N$, and each $NY_{N,\phi}$ is a disjoint product of a finite open covering of $X$ by affine standard smooth schemes. Thus, it is also affine standard smooth, and as explained in the introduction we may lift them to standard smooth schemes $NE_{N,\phi}$ over $W(k)$. This gives a cartesian diagram
$$\xymatrix{Y_N=\bigsqcup\limits_{\phi: [N]\surj [m]} NY_{N,\phi} \ar[r]\ar[d] & E_N:=\bigsqcup\limits_{\phi: [N]\surj [m]} NE_{N,\phi}\ar[d]\\
\Spec k\ar[r] & \Spec W(k).}$$
Then, we may form the $N$-truncated simplicial $W(k)$-scheme 
$$F_{\bullet\leq N}:=\Gamma_N^{W(k)}(E_N)_{\bullet\leq N}$$
and by Lemma \ref{nak6.6.1} we get an $N$-truncated special frame
\begin{align}\label{the frame}(Y_{\bullet\leq N},F_{\bullet\leq N}).\end{align}

\section{The comparison theorem}

We outline the formation of the comparison map, we want to work on a special hypercovering $X_\bullet\to X$, and for some $N$ to construct the $N$-truncated simplicial frame $(X_{\bullet\leq N},F_{\bullet\leq N})$ as in (\ref{the frame}), which will in turn give us $N$-truncated dagger spaces and rigid frames
$$]X_{\bullet\leq N}[^\dagger_{\Fa_{\bullet\leq N}},$$
$$(X_{\bullet\leq N}, Y_{\bullet\leq N}, \Qa_{\bullet\leq N}).$$ 
We will use the functorial quasi-isomorphisms
$$\t{sp}_*\Omega^\bullet_{]X_{\bullet\leq N}[^\dagger_{\Fa_{\bullet\leq N}}}\to W^\dagger\Omega^\bullet_{X_{\bullet\leq N},k}\otimes \Q$$
from \cite{Davis2011} to give us a quasi-isomorphism 
$$R\Gamma(X_{\bullet\leq N},\t{sp}_*\Omega^\bullet_{]X_{\bullet\leq N}[^\dagger_{\Fa_{\bullet\leq N}}})\stackrel{\sim}{\to} R\Gamma(X_{\bullet\leq N}, W^\dagger\Omega^\bullet_{X_{\bullet\leq N},k})\otimes \Q.$$
We will then prove vanishing of the higher cohomologies of the $\t{Rsp}_*\Omega^\bullet_{]X_{m}[^\dagger_{\Fa_{m}}}$ for this particular special frames in Proposition \ref{vanish} and show that
$$R\Gamma(X_{\bullet\leq N},R\t{sp}_*\Omega^\bullet_{]X_{\bullet\leq N}[^\dagger_{\Fa_{\bullet\leq N}}})$$
calculates $\Rig(X/K)$ for $N$ large enough. This last part, and showing independence of choices made and functoriality are quite, and are motivated by \cite{Nakkajima2012}.

In the course of the proof, we will need some tools from \cite{Davis2011} in order to compare special frames. The first result is proven in the proof of Proposition 4.35 and the second result is Proposition 4.37.
\begin{prop}\label{Dav sp_fr}
\begin{itemize}
\item[i)] Given a map of special frames
$$\xymatrix{X\ar[r]\ar@{=}[d] & F'\ar[d]\\ X\ar[r]& F}$$
with the right vertical map being \et ale, then we get a natural isomorphism of dagger spaces
$$]X[^\dagger_{\Fa'}\cong ]X[^\dagger_{\Fa}.$$
\item[ii)] Given a special frame $(X,F\times\A^n_{W(k)})$ for any $n$, such that the map $X\to \A^n_{W(k)}$ factors through the origin, there exists a natural quasi-isomorphism
$$\t{Rsp}_*\Omega^\bullet_{]X[^\dagger_{\Fa}}\stackrel{\sim}{\to}\t{Rsp}_*\Omega^\bullet_{]X[^\dagger_{\Fa\times \A^n_{\Wa}}}.$$
\end{itemize}
\end{prop}

When proving the comparison, we will need to show vanishing of the higher cohomologies of $\t{Rsp}_*\Omega^\bullet_{]Y_m[^\dagger_{\Fa_m}}$ for $0\leq m\leq N$, where $(Y_m,F_m)$ are the special frames constructed in sections \ref{hypercovering} and \ref{frame}. The above proposition will allow us to reduce it to the following theorem, which follows from the proof of \cite[Theorem 1.10]{Berthelot1997}:
\begin{prop}\label{Ber1.10} Given a special frame $(X,F)$, where $F$ is a lifting of $X$ over $W(k)$, then
$$\t{R}^i\t{sp}_*\Omega^\bullet_{]X[^\dagger_{\Fa}}=0 \text{ for i>0}.$$
\end{prop}

We now prove a key ingredient of the comparison theorem:

\begin{prop}\label{vanish} Given an $N$-truncated simplicial frame $(Y_{\bullet\leq N},F_{\bullet\leq N})$ as in (\ref{the frame}), for $0\leq m\leq N$ and $i>0$,
$$\t{R}^i\t{sp}_*\Omega^\bullet_{]Y_m[^\dagger_{\Fa_m}}=0.$$
\end{prop}
\begin{proof}
Pick any $0\leq m\leq N$. By splitness of $Y_\bullet$, we may write
$$Y_m=\bigsqcup\limits_{\phi: [N]\surj [m]} NY_{m,\phi}.$$
Fix some degeneracy map $\sigma: [N]\surj [m]$. Then, by construction of $\Gamma_N^{W(k)}(-)$, we have a commutative diagram
$$\xymatrix{Y_m\ar[d]_{Y_\bullet(\sigma)}\ar[r] & F_m=\prod\limits_{\alpha: [N]\to [m]} E_\alpha\ar[d]^(0.6){p_\sigma}\\
Y_N\ar[r] & E_N=E_\sigma}$$
where $E_\sigma=E_N$ was defined in section \ref{frame}, and $p_\sigma$ is the projection, and both horizontal maps and the left vertical map are closed immersions. This gives us a closed immersion 
$$Y_m\into E_\sigma.$$

Let 
$$F'_m:=\prod\limits_{\alpha: [N]\to [m], \alpha\neq \sigma} E_\alpha,$$
so $F_m=F'_m\times E_\sigma$. Then, since each of the $E_\alpha$ are standard smooth schemes over $W(k)$, so is their product, and we may get an \et ale morphism
$$F'_m\to \A^n_{W(k)}$$
for some $n$. Thus, considering the commutative diagram
$$\xymatrix{Y_m\ar[r]\ar@{=}[d] & F'_m\times E_\sigma\ar[d]\ar@{=}[r] & F_m\\ Y_m\ar[r]& \A^n_{W(k)}\times E_\sigma}$$
where the right vertical morphism is \et ale, using Proposition \ref{Dav sp_fr}.i) we may reduce to the case of the special frame $(Y_m,\A^n_{W(k)}\times E_\sigma)$. Furthermore, we may assume that the map $Y_m\to \A^n_{W(k)}$ factors through the origin. To see this, write $Y_m=\Spec (A)$ and $E_\sigma=\Spec B$, so $\A^n_{W(k)}\times E_\sigma=\Spec B[T_1,...,T_n]$. Then, since $B\surj A$ is surjective (as $Y_m\to E_\sigma$ is a closed immersion), we may pick $b_1,...,b_n\in B$ which map to the images of $T_1,...,T_n$ respectively in $A$, and replace $T_i$ by $T'_i:=T_i-b_i$, giving a special frame
$$(Y_m,\Spec B[T'_1,...,T'_n])=(Y_m,\A^n_{W(k)}\times E_\sigma)$$
factoring through the origin. Thus, by Proposition \ref{Dav sp_fr}.ii), we reduce the proof to the special frame $(Y_m,E_\sigma)$. 

Now, since
$$]Y_m[^\dagger_{\mathcal{E}_\sigma}=]\bigsqcup NY_{m,\phi}[^\dagger_{\mathcal{E}_\sigma}\cong \bigsqcup ]NY_{m,\phi}[^\dagger_{\mathcal{E}_\sigma}$$
we may reduce to studying the special frames $(NY_{m,\phi},E_\sigma)$ for any $\phi:[m]\surj [k]$ and $0\leq k\leq m$. But notice that by the construction of the frame, for any $\phi:[m]\surj [k]$, we have a commutative diagram
$$\xymatrix{NY_{m,\phi}\ar@{}[d]|*[@]{\cong}\ar@{}[r]|-*[@]{\subset} & Y_m\ar@{^{(}->}[d]^{Y_\bullet(\sigma)}\\
NY_{N,\phi\circ \sigma}\ar@{}[r]|-*[@]{\subset} & Y_N\ar@{^{(}->}[r] & E_\sigma\ar@{=}[r] & E_N\ar@{=}[r] & \bigsqcup NE_{N,\psi}}$$
where $\alpha$ vary over all morphisms $\psi: [N]\surj [k']$ with $0\leq k'\leq N$, and the composite map $NY_{m,\phi}\to E_N$ is the map giving the special frame. Thus,
$NY_{m,\phi}$ is isomorphic to $NE_{N,\phi\circ \sigma}\subset E_N$, and therefore
$$\t{sp}\inv(NY_{m,\phi})=]NY_{m,\phi}[^\dagger_{\mathcal{E}_\sigma}=]NY_{m,\phi}[^\dagger_{N\mathcal{E}_{N,\phi\circ \sigma}},$$
which reduces the proof to the case of the special frame $(NY_{m,\phi},NE_{N,\phi\circ \sigma})$.

But by construction, $NE_{N,\phi\circ \sigma}$ is a smooth lift of $NY_{N,\phi\circ \sigma}\cong NY_{m,\phi}$ over $W(k)$, and thus we can apply Proposition \ref{Ber1.10} to complete the proof.
\end{proof}

We will need the following to deal with $N$-truncations, which basically says that for some large enough $N$, we only need the $N$-skeleton in the calculations of cohomologies on simplicial objects (such as for rigid cohomology and overconvergent Witt de-Rham). For a complex $A^\bullet$ of $K$ vector spaces, and any $h$, consider the $h$-truncated complex
$$\tau_{\leq h}(A^\bullet)^i=\begin{cases} A^i \text{ if }i<h\\ \ker(A^h\to A^{h+1}) \text{ if } i>h \\ 0 \text{ else.}\end{cases}$$
For a double complex $A^{\bullet\bullet}$, let $\tau_{\leq h}^{(1)}(A^{\bullet q}):=\tau_{\leq h}(A^{\bullet q})$, and let $s: C(K)\to K$ be the total complex map.

\begin{lem}\label{N-trunc} \cite[Lemma 2.2]{Nakkajima2012} Consider a double complex $A^{\bullet,\bullet}$ of $K$ vector spaces such that $A^{p,q}=0$ for $p<0$ or $q<0$. Given any 
\begin{align}\label{N,h}N>\max\{i+(h-i+1)(h-i+2)/2 \mid 0\leq i\leq h\}=(h+1)(h+2)/2,\end{align}
 the natural maps $s(\tau^{(1)}_{\leq N}(A^{\bullet\bullet})\to s(A^{\bullet\bullet})$ induce a quasi-isomorphism
$$\tau_{\leq h}(s(\tau^{(1)}_{\leq N}(A^{\bullet\bullet}))\stackrel{\sim}{\to} \tau_{\leq h}(s(A^{\bullet\bullet})).$$
\end{lem}

From this, and the formation of the spectral sequence for cohomology on simplicial objects, it follows for example that for some simplicial rigid frame $(Z_\bullet,\overline{Z}_\bullet,\Za_\bullet)$, and $h$ and $N$ as in (\ref{N,h}), we get natural quasi-isomorphisms
$$\tau_{\leq h}R\Gamma(]\overline{Z}_{\bullet\leq N}[_{\Za_{\bullet\leq N}},j^\dagger\Omega^\bullet_{]\overline{Z}_{\bullet\leq N}[_{\Za_{\bullet\leq N}}})\stackrel{\sim}{\to} \tau_{\leq h}R\Gamma(]\overline{Z}_{\bullet}[_{\Za_{\bullet}},j^\dagger\Omega^\bullet_{]\overline{Z}_{\bullet}[_{\Za_{\bullet}}})$$
and that for a smooth simplicial scheme $X_\bullet$,
$$\tau_{\leq h}R\Gamma(X_\bullet,W^\dagger\Omega^\bullet_{X_{\bullet/k}})\stackrel{\sim}{\to} \tau_{\leq h}R\Gamma(X_{\bullet\leq N},W^\dagger\Omega^\bullet_{X_{\bullet\leq N/k}}).$$

This is useful by the following theorem of vanishing of rigid cohomology:

\begin{thm}\label{vanish rigid}\cite[Theorem 6.4.1]{Tsuzuki2004}  Given a scheme $X$ over $k$, there exists an integer $c$ such that for $i>c$, $H^i_{rig}(X/K)=0$. 
\end{thm}

We can now prove the main comparison theorem:

\begin{thm}\label{comparison theorem}
Given a smooth scheme $X$ over $k$, there exists a functorial quasi-isomorphism
$$\t{R}\Gamma_{rig}(X/K)\stackrel{\sim}{\to} \t{R}\Gamma(X, W^\dagger \Omega^\bullet_{X/k})\otimes \Q.$$
\end{thm}
\begin{proof}
We form a special hypercovering
$$X_\bullet\to X.$$
For any $h$, we form an $N$-truncated special frame  
$$(X_{\bullet\leq N},F_{\bullet\leq N})$$
as explained in section \ref{frame}. Then, for $0\leq m\leq N$ and $i>0$, 
$$\t{R}^i\t{sp}_*\Omega^\bullet_{]X_m[^\dagger_{\Fa_m}}=0$$
by Proposition \ref{vanish}. 

Next, by \cite{Davis2011} we have natural isomorphisms of Zariski sheaves
$$\t{sp}_*\Omega^\bullet_{]X_m[^\dagger_{\Fa_m}}\stackrel{\sim}{\to} W^\dagger\Omega^\bullet_{X_m/k}\otimes \Q$$
giving a $N$-truncated simplicial version
$$\t{sp}_*\Omega^\bullet_{]X_{\bullet\leq N}[^\dagger_{\hat{F}_{\bullet\leq N}}}\stackrel{\sim}{\to} W^\dagger\Omega^\bullet_{X_{\bullet\leq N}/k}\otimes \Q.$$
Then, by vanishing of the higher $\t{R}^i\t{sp}_*$, and applying $\t{R}\Gamma(X,\t{R}\epsilon_*(-))$ to both sides we obtain
\begin{align}\label{eq1}R\Gamma(]X_{\bullet\leq N}[^\dagger_{\Fa_{\bullet\leq N}},\Omega^\bullet_{]X_{\bullet\leq N}[^\dagger_{\Fa_{\bullet\leq N}}})\stackrel{\sim}{\to} R\Gamma(X_{\bullet\leq N}, W^\dagger\Omega^\bullet_{X_{\bullet\leq N}})\otimes \Q.\end{align}

Then, by Lemma \ref{N-trunc} and $X_\bullet\to X$ being an \et ale hypercovering, we get
\begin{align}\label{eq6}\tau_{\leq h}(R\Gamma(X_{\bullet\leq N}, W^\dagger\Omega^\bullet_{X_{\bullet\leq N}}))\stackrel{\sim}{\to}\tau_{\leq h}(R\Gamma(X_{\bullet}, W^\dagger\Omega^\bullet_{X_{\bullet}/k}))\stackrel{\sim}{\leftarrow} \tau_{\leq h}(R\Gamma(X,W^\dagger\Omega^\bullet_{X/k})).\end{align}

To complete the proof, we will need to show that the left hand side of (\ref{eq1}) calculates $\Rig(X/K)$ for $h=c$ as in Proposition \ref{vanish rigid}, that the isomorphism in $D_+(K)$ with $\Rig(X/K)$ is independent of choices made, and that it can be done functorially.\\

\textit{The left hand side of (\ref{eq1}) calculates $\Rig(X/K)$:}

We construct an isomorphism in $D_+(K)$. Firstly, from the $N$-truncated simplicial special frame $(Y_{\bullet\leq N},F_{\bullet\leq N})$ we construct a simplicial rigid frame 
$$(X_{\bullet\leq N},Y_{\bullet\leq N},\Qa_{\bullet\leq N})$$
as outlined in Section \ref{special + dagger}. Then, since the construction of the dagger spaces $]Y_m[^\dagger_{\hat{F}_m}$ are functorial in $(Y_m,F_m)$, by a simplicial version of \cite[Theorem 5.1]{Grosse-Klonne2000} we get quasi-isomorphisms
\begin{align}\label{eq Grosse-Klonne}R\Gamma(]X_{\bullet\leq N}[^\dagger_{\Fa_{\bullet\leq N}},\Omega^\bullet_{]X_{\bullet\leq N}[^\dagger_{\Fa_{\bullet\leq N}}})\stackrel{\sim}{\to}R\Gamma(]Y_{\bullet\leq N}[_{\Qa_{\bullet\leq N}},j^\dagger\Omega^\bullet_{]Y_{\bullet\leq N}[_{\Qa_{\bullet\leq N}}}) .\end{align}

Now, we construct a complex which calculates $\Rig(X/K)$ similar to the proof of \cite[Theorem 11.1.1]{Chiarellotto2003}. We take a compactification
$$X\into \overline{X}$$
for some proper $k$-scheme $\overline{X}$. Next, take a finite open affine covering of $\overline{X}$, and let $\overline{U}$ be its disjoint union, and $U:=\overline{U}\times_{\overline{X}} X.$ Since $\overline{U}$ is affine, we may take a closed embedding into some smooth formal $\Wa$-scheme $\mathcal{U}$. Set
$$(U_\bullet,\overline{U}_\bullet,\mathcal{U}_\bullet):=(\cosk_0^X(U),\cosk_0^{\overline{X}}(\overline{U}),\cosk_0^{\Wa}(\mathcal{U})).$$
This is a universal de Rham descendable hypercovering of $(X,\overline{X})$ (in the sense of \cite{Chiarellotto2003}), and thus by independence of choice of compactification and such a hypercovering (see \cite[Proposition 10.4.3, Corollary 10.5.4]{Chiarellotto2003}) we may define 
$$\Rig(X/K):=R\Gamma(]\overline{U}_\bullet[_{\mathcal{U}_\bullet},j^\dagger\Omega^\bullet_{]\overline{U}_\bullet[_{\mathcal{U}_\bullet}}).$$

Next, by Proposition \ref{simplicial compactification} we may construct a proper hypercovering $\overline{X}_\bullet$ over $\overline{X}$ with an open embedding $X_\bullet\into \overline{X}_\bullet$, making $(X_\bullet, \overline{X}_\bullet)$ into a \et ale-proper hypercovering of $(X,\overline{X})$. By \cite[Lemma 7.2.3]{Chiarellotto2003} and both hypercoverings being preserved by base change, we get an \et ale-proper hypercoverings
\begin{align}\label{et-prop}(\cosk_N^X(\sk^X_N(X_\bullet))\times_X U, \cosk_N^{\overline{X}}(\sk_N^{\overline{X}}(\overline{X}_\bullet))\times_{\overline{X}} \overline{U})\to (U,\overline{U}).\end{align}
Then, using \cite[Proposition 11.5.1]{Chiarellotto2003} we may construct an \et ale-proper hypercovering $(V_\bullet,\overline{V}_\bullet,\Va_\bullet)$ of $(U,\overline{U})$ such that $(V_\bullet,\overline{V}_\bullet)$ is a refinement of (\ref{et-prop}).

Let $(V_{\bullet,\bullet},\overline{V}_{\bullet,\bullet},\mathcal{V}_{\bullet,\bullet})$ be 
$$V_{m,n}=\cosk_0^{Y_n}(V_n)_m, \, \overline{V}_{m,n}=\cosk_0^{\overline{Y}_n}(\overline{V}_n)_m, \, \Va_{m,n}=\cosk_0^\Wa(\Va_n\hat{\times}_\Wa \mathcal{U}_n)_m.$$
Defining the simplicial maps as in \cite[Proposition 11.5.4]{Chiarellotto2003} we have 
\begin{itemize}
\item[(*)] $(V_{m,\bullet},\overline{V}_{m,\bullet},\mathcal{V}_{m,\bullet})$ is an \et ale-proper hypercovering of $(U_m,\overline{U}_m,\mathcal{U}_m)$ for any $m$;
\item[(**)] $(V_{\bullet,n},\overline{V}_{\bullet,n},\mathcal{V}_{\bullet,n})$ is a universally de Rham descendable hypercovering of $(Y_n,\overline{Y}_n)$ for any $n$. 
\end{itemize}

Considering the $(\infty,N)$-truncated version $(V_{\bullet,\bullet\leq N},\overline{V}_{\bullet,\bullet\leq N},\mathcal{V}_{\bullet,\bullet\leq N})$, we get a morphism
\begin{align}\label{eq2}\Rig(X/K)=R\Gamma(]\overline{U}_\bullet[_{\mathcal{U}_{\bullet}},j^\dagger\Omega^\bullet_{]\overline{U}_\bullet[_{\mathcal{U}_{\bullet}}})\to R\Gamma(]\overline{V}_{\bullet,\bullet\leq N}[_{\mathcal{V}_{\bullet,\bullet\leq N}},j^\dagger\Omega^\bullet_{]\overline{V}_{\bullet,\bullet\leq N}[_{\mathcal{V}_{\bullet,\bullet\leq N}}}).\end{align}

We claim this induces becomes a quasi-isomorphism upon applying $\tau_{\leq h}$. To see this, compare the spectral sequences
$$E_1^{pq}=H^q(]\overline{U}_p[_{\mathcal{U}_{p}},j^\dagger\Omega^\bullet_{]\overline{U}_p[_{\mathcal{U}_{p}}})\Rightarrow H^{p+q}_{rig}(X/K),$$
$$E_1^{pq}=H^q(]\overline{V}_{p,\bullet\leq N}[_{\mathcal{V}_{p,\bullet\leq N}},j^\dagger\Omega^\bullet_{]\overline{V}_{p,\bullet\leq N}[_{\mathcal{V}_{p,\bullet\leq N}}})\Rightarrow H^{p+q}_{rig}(]\overline{V}_{\bullet,\bullet\leq N}[_{\mathcal{V}_{\bullet,\bullet\leq N}},j^\dagger\Omega^\bullet_{]\overline{V}_{\bullet,\bullet\leq N}[_{\mathcal{V}_{\bullet,\bullet\leq N}}}).$$
and notice that for $q\leq h$, by Lemma \ref{N-trunc},
$$H^q(]\overline{V}_{p,\bullet\leq N}[_{\mathcal{V}_{p,\bullet\leq N}},j^\dagger\Omega^\bullet_{]\overline{V}_{p,\bullet\leq N}[_{\mathcal{V}_{p,\bullet\leq N}}})\cong H^q(]\overline{V}_{p,\bullet}[_{\mathcal{V}_{p,\bullet}},j^\dagger\Omega^\bullet_{]\overline{V}_{p,\bullet}[_{\mathcal{V}_{p,\bullet}}})\cong H^q(]\overline{U}_p[_{\mathcal{U}_{p}},j^\dagger\Omega^\bullet_{]\overline{U}_p[_{\mathcal{U}_{p}}})$$
where the last isomorphism comes from (*). So we get a quasi-isomorphism after applying $\tau_{\leq h}$ to (\ref{eq2}).

By Lemma \ref{N-trunc}, we get
\begin{align}\label{eq3}\tau_{\leq h}\Rig(X/K)\stackrel{\sim}{\to} 
\tau_{\leq h}(R\Gamma(]\overline{V}_{\bullet,\bullet\leq N}[_{\mathcal{V}_{\bullet,\bullet\leq N}},j^\dagger\Omega^\bullet_{]\overline{V}_{\bullet,\bullet\leq N}[_{\mathcal{V}_{\bullet,\bullet\leq N}}})).\end{align}

Now, we want to compare the right hand side of (\ref{eq3}) to $\tau_{\leq h}$ applied to the right hand side of (\ref{eq1}). Let $W_{m,m}:=V_{m,n}$, let $\overline{W}_{m,n}$ be the scheme theoretical closure of $$V_{m,n}=V_{m,n}\times_{X_n} X_n\to \overline{V}_{m,n}\times_k Y_n$$ 
and $\Ra_{m,n}:=\Va_{m,n}\hat{\times}_{\Wa} \Qa_n$ with the obvious maps. This gives maps
\begin{align}\label{eq4}\xymatrix{R\Gamma(]\overline{V}_{\bullet,\bullet\leq N}[_{\mathcal{V}_{\bullet,\bullet\leq N}},j^\dagger\Omega^\bullet_{]\overline{V}_{\bullet,\bullet\leq N}[_{\mathcal{V}_{\bullet,\bullet\leq N}}})\ar[d]\\
R\Gamma(]\overline{W}_{\bullet,\bullet\leq N}[_{\mathcal{R}_{\bullet,\bullet\leq N}},j^\dagger\Omega^\bullet_{]\overline{W}_{\bullet,\bullet\leq N}[_{\mathcal{R}_{\bullet,\bullet\leq N}}})\\ R\Gamma(]Y_{\bullet\leq N}[_{\mathcal{\Qa}_{\bullet\leq N}},j^\dagger\Omega^\bullet_{]Y_{\bullet\leq N}[_{\mathcal{\Qa}_{\bullet\leq N}}})\ar[u]}
\end{align}
which we claim become quasi-isomorphisms once we apply $\tau_{\leq h}$. To see this, consider the induced map on spectral sequences
\begin{align}\label{eq5}\xymatrix{E_1^{pq}=H^q(]\overline{V}_{\bullet,p}[_{\mathcal{V}_{\bullet,p}},j^\dagger\Omega^\bullet_{]\overline{V}_{\bullet,p}[_{\mathcal{V}_{\bullet,p}}})\ar@{=>}[r]\ar[d] &        H^{p+q}(]\overline{V}_{\bullet,\bullet\leq N}[_{\mathcal{V}_{\bullet,\bullet\leq N}},j^\dagger\Omega^\bullet_{]\overline{V}_{\bullet,\bullet\leq N}[_{\mathcal{V}_{\bullet,\bullet\leq N}}})\ar[d]\\
''E_1^{pq}=H^q(]\overline{W}_{\bullet,p}[_{\mathcal{R}_{\bullet,p}},j^\dagger\Omega^\bullet_{]\overline{W}_{\bullet,p}[_{\mathcal{R}_{\bullet,p}}})  \ar@{=>}[r] &           H^{p+q}(]\overline{W}_{\bullet,\bullet\leq N}[_{\mathcal{R}_{\bullet,\bullet\leq N}},j^\dagger\Omega^\bullet_{]\overline{W}_{\bullet,\bullet\leq N}[_{\mathcal{R}_{\bullet,\bullet\leq N}}})\\ 
'E_1^{pq}=H^q(]Y_{p}[_{\mathcal{\Qa}_{p}},j^\dagger\Omega^\bullet_{]Y_{p}[_{\mathcal{\Qa}_{p}}})\ar@{=>}[r]\ar[u] &   H^{p+q}(]Y_{\bullet\leq N}[_{\mathcal{\Qa}_{\bullet\leq N}},j^\dagger\Omega^\bullet_{]Y_{\bullet\leq N}[_{\mathcal{\Qa}_{\bullet\leq N}}})\ar[u]}
\end{align}
for $p\leq N$.

The top left vertical arrow is an isomorphism by the independence of choice of compactification of the $V_{m,p}$ in \cite[Proposition 8.3.5]{Chiarellotto2003}, since for any $p,m$, we have a morphism of rigid frames
$$\xymatrix{V_{m,p}\ar@{=}[d] \ar[r] & \overline{W}_{m,p}\ar[d]\ar[r] & \Va_{m,p}\ar[d]\\
V_{m,p}\ar[r] & \overline{V}_{m,p}\ar[r] & \Pa_{m,p}}$$
with middle and right vertical maps proper and smooth respectively. 

Next, notice that both $E_1^{pq}$ and $'E_1^{pq}$ calculate $H^q_{rig}(Y_p/K)$ by (**), which is finitely generated. Thus, the bottom left vertical map in (\ref{eq5}) must also be an isomorphism.

Summarizing all of the above, we get isomorphisms in $D_+(K)$
\begin{align}\label{comparison maps}\xymatrix{\tau_{\leq h}(\Rig(X/K))\ar[r]^(0.37)\sim_(0.37){(\ref{eq2})} &
\tau_{\leq h}(R\Gamma(]\overline{V}_{\bullet,\bullet\leq N}[_{\mathcal{\mathcal{V}_{\bullet,\bullet\leq N}}_{\bullet\leq N}},j^\dagger\Omega^\bullet))\ar[r]^\sim_{(\ref{eq4})} &
\tau_{\leq h}(R\Gamma(]\overline{W}_{\bullet,\bullet\leq N}[_{\mathcal{\mathcal{R}_{\bullet,\bullet\leq N}}_{\bullet\leq N}},j^\dagger\Omega^\bullet))\\
&&
\tau_{\leq h}(R\Gamma(]Y_{\bullet\leq N}[_{\mathcal{\Qa}_{\bullet\leq N}},j^\dagger\Omega^\bullet))\ar[u]^\sim_{(\ref{eq4})}\\
&&
\tau_{\leq h}(R\Gamma(]X_{\bullet\leq N}[^\dagger_{\mathcal{\Fa}_{\bullet\leq N}},\Omega^\bullet))\ar[d]_\sim^{(\ref{eq1})}\ar[u]^\sim_{(\ref{eq Grosse-Klonne})}\\
\tau_{\leq h}(R\Gamma(X,W^\dagger\Omega^\bullet_{X/k}))\ar[r]^\sim_{(\ref{eq6})} & 
\tau_{\leq h}(R\Gamma(X_{\bullet},W^\dagger\Omega^\bullet_{X_{\bullet/k}})) &
\tau_{\leq h}(R\Gamma(Y_{\bullet\leq N},W^\dagger\Omega^\bullet_{X_{\bullet\leq N/k}}))\ar[l]_\sim^{(\ref{eq6})} }\end{align}
where we have ommitted subscripts in the $j^\dagger\Omega^\bullet$ and $\Omega^\bullet$.

By Theorem \ref{vanish rigid}, there exists a $c$ such that $H^i_{rig}(X/K)=0$ for $i>c$. Varying $h$ (and $N$) we see that 
$$\tau_{\leq c}(\Rig(X/K))\stackrel{\sim}{\to} \Rig(X/K), \,\,\,\, \tau_{\leq c}(R\Gamma(X,W^\dagger\Omega^\bullet_{X/k}))\stackrel{\sim}{\to} R\Gamma(X,W^\dagger\Omega^\bullet_{X/k})$$ 
so we may set $h=c$ in (\ref{comparison maps}) and drop the truncation terms, giving us an isomorphism in $D_+(K)$
$$\Rig(X/K)\cong R\Gamma(X,W^\dagger\Omega^\bullet_{X/k}).$$

\textit{Independence of choices:}

We must prove independence of the choices of the special hypercovering $X_\bullet$, $c$ as in Theorem \ref{vanish rigid}, $N$ satisfying (\ref{N,h}) for $c$, lifting $E_N$ of $X_N$ over $W(k)$ and its immersion into affine space and projective space $\A^r_{W(k)}$ and $P=\P^r_{W(k)}$, the compactification $\overline{X}$, the Zariski covering $\overline{U}$ of $\overline{X}$ and its closed immersion into $\mathcal{U}$, and the refinement $(V_\bullet,\overline{V}_\bullet,\mathcal{V}_\bullet)$ of $(X_\bullet\times_X U, \overline{X}_\bullet\times_{\overline{X}}\overline{U})$ over $(U,\overline{U})$.

\textit{1) Independence of $X_\bullet$, $E_N\into \A^r_W\into P=\P^r_W$:}

Suppose we have two choices 
$$(X_\bullet^i, E^i_N\into \A^{r_i}_W\into P^i=\P^{r_i}_W, \,i=1,2),$$
with all the other choices the same. 

By Proposition \ref{sp hypercovering}, there exists a special hypercovering $X^{12}_\bullet\to X$ refining $X^1_\bullet$ and $X^2_\bullet$. We can choose some lifting $E_{N}^{12}$ of $X_{N}^{12}$ over $W(k)$ fitting into the diagram of special frames 
$$\xymatrix{ & (X^1_N,E^1_N)\\
(X^{12}_N,E^{12}_N)\ar[ur]\ar[dr]\\
& (X^2_N,E^2_N).}$$

This will give us morhpisms in the $N$-truncated simplicial rigid frames $$(X_{\bullet\leq N}^{12},Y_{\bullet\leq N}^{12},\Qa^{12}_{\bullet\leq N})\to (X_{\bullet\leq N}^{i},Y_{\bullet\leq N}^{i},\Qa^{i}_{\bullet\leq N})$$ for $i=1,2$ by functoriality of the $\Gamma_N^{W(k)}(-)$ functor. 

Next, by Proposition \ref{simplicial compactification}.ii) (with the argument slightly modified to involve a triple fiber product) we may form a proper \et ale hypercovering $(X_{\bullet},\overline{X}^{12}_\bullet)$ of $(X,\overline{X})$ with maps to $(X_\bullet,\overline{X}^i\bullet)$ for $i=1,2$. 

Then, by the proof of \cite[Proposition 11.5.2]{Chiarellotto2003} we may choose refinements $(V'_\bullet,\overline{V}'_\bullet,\Va'_\bullet)$ and $(V_\bullet'',\overline{V}''_\bullet,\Va_\bullet'')$ of $(U\times_X X^{12},\overline{U}\times_{\overline{X}} \overline{X}^{12})$ over $(U,\overline{U})$ fitting into diagrams
$$\xymatrix{(V'_\bullet,\overline{V}'_\bullet,\Va'_\bullet)\ar[d]\ar[r] & (V^1_\bullet,\overline{V^1}_\bullet,\Va^1_\bullet)\ar[d]\\ 
(U\times_X\cosk_N^X(X^{12}_{\bullet\leq N}),\overline{U}\times_{\overline{X}}\cosk_N^{\overline{X}}(\overline{X}^{12}_{\bullet\leq N}))\ar[r] & (U\times_X \cosk_N^X(X^{1}_{\bullet\leq N}),\overline{U}\times_{\overline{X}}\cosk_N^{\overline{X}}(\overline{X}^{1}_{\bullet\leq N}))\\
(V''_\bullet,\overline{V}''_\bullet,\Va''_\bullet)\ar[d]\ar[r] & (V^2_\bullet,\overline{V^2}_\bullet,\Va^2_\bullet)\ar[d]\\
(U\times_X\cosk^{X}_N(X^{12}_{\bullet\leq N}),\overline{U}\times_{\overline{X}}\cosk_N^{\overline{X}}(\overline{X}^{12}_{\bullet\leq N}))\ar[r] & (U\times_X \cosk_N^{X}(X^{2}_{\bullet\leq N}),\overline{U}\times_{\overline{X}}\cosk_N^{\overline{X}}(\overline{X}^{2}_{\bullet\leq N}))}$$
where the vertical maps are only seen as morphisms of pairs. Taking $(V^{12}_\bullet,\overline{V}^{12}_\bullet,\Va^{12}_\bullet)$ to be the fiber product of these two refinements, we get by \cite[Proposition 11.5.1.(2)]{Chiarellotto2003} that this is also a refinements of $(U\times_X X^{12},\overline{U}\times_{\overline{X}} \overline{X}^{12})$ over $(U,\overline{U})$, with maps to  $(V^{i}_\bullet,\overline{V}^{i}_\bullet,\Va^{i}_\bullet)$ for $i=1,2$ compatible with the other maps. 

Next, having taking closures $\overline{W}^i_{m,n}$ of $V^i_{m,n}$ in $\overline{V}^i_{m,n}\times_k Y^i_n$ with closed immersion into $\Ra^i_{m,n}:=\Va^i_{m,n}\hat{\times}_{\Wa} \Qa^{i}_{n}$ for all $m$, $n\leq N$, and $i=1,2,12$. We have the diagram
$$\xymatrix{V^1_{m,n}\ar[r]& \overline{W}^1_{m,n}\ar[r]& \overline{V}^1_{m,n}\times_k Y_n^{1}\\
V^{12}_{m,n}\ar[r]\ar[u]\ar[d]& \overline{W}^{12}_{m,n}\ar[r]& \overline{V}^{12}_{m,n}\times_k Y_n^{12}\ar[u]\ar[d]\\
V^2_{m,n}\ar[r]& \overline{W}^2_{m,n}\ar[r]& \overline{V}^2_{m,n}\times_k Y_n^{2}}$$
which by universal property of closures of a map, gives a factorization
$$V^{12}_{m,n}\to \overline{W}^{12}_{m,n}\to ((\overline{V}^{12}_{m,n}\times_k Y_n^{12})\times_{(\overline{V}^{1}_{m,n}\times_k Y_n^{1})} \overline{W}^1_{m,n})\times_{(\overline{V}^{2}_{m,n}\times_k Y_n^{2})} \overline{W}^{12}_{m,n} \to V^{12}_{m,n}\times_k Y^{12}_{n}.$$

This gives us a diagaram of closed immersions into smooth formal schemes
$$\xymatrix{ & (\overline{W}^{1}_{m,n},\Ra^1_{m,n})\\
(\overline{W}^{12}_{m,n},\Ra^{12}_{m,n})\ar[ur]\ar[dr]\\
& (\overline{W}^{2}_{m,n},\Ra^2_{m,n}).}$$

Then, we have maps from the $i=1,2$ versions of (\ref{comparison maps}) to a common one with superscript $12$, where all maps are quasi-isomorphisms.

\textit{2) Independence of $\overline{X}, \overline{U}, \mathcal{U}$:}

Given choices of $\overline{X}^i, \overline{U}^i$ and $\mathcal{U}^i$ for $i=1,2$, fix the other choices. We may let $\overline{X}^{12}$ be the closure of $X$ in $\overline{X}^1\times_k \overline{X}^2$. Then, take an open affine covering of $\overline{X}^{12}$ such that each affine is projected into one of the opens in $\overline{X}^i$ giving $\overline{U}^i$ for $i=1,2$, so that their disjoint union $\overline{U}^{12}$ has compatible morphisms to $\overline{U}^i$ for $i=1,2$. Set $U^{12}=\overline{U}^{12}\times_{\overline{X}^{12}} X$.

Now, having refinements $(V^i_\bullet,\overline{V}^i_\bullet,\Va^i_\bullet)$ of $(X_\bullet\times_X U^i, \overline{X}_\bullet\times_{\overline{X}^i} \overline{U}^i)$ over $(U^i,\overline{U}^i)$ for $i=1,2$, we use the same argument used above to construct $(V^{12}_\bullet,\overline{V}^{12}_\bullet,\Va^{12}_\bullet)$ from $(V^{'}_\bullet,\overline{V}^{'}_\bullet,\Va^{'}_\bullet)$ and $(V^{''}_\bullet,\overline{V}^{''}_\bullet,\Va^{''}_\bullet)$, which also gives us $(V^{12}_{\bullet,\bullet},\overline{V}^{12}_{\bullet,\bullet},\Va^{12}_{\bullet,\bullet})$. We also use the same argument to obtain $\overline{W}^{12}_{\bullet,\bullet\leq N}$, $\Ra^{12}_{\bullet,\bullet\leq N}$. As before, this gives independence of these choices.

\textit{3) Independence of $c,N$:}

The independence of $c$ is clear. Suppose we have some $N^1\leq N^2$ satisfying (\ref{N,h}). Then, given choices of embeddings $X_{N_i}\into E_{N_i}\into \A^{r_i}_W\into \P^{r_1}_W$ for $i=1,2$, we will get $Y^i_{\bullet\leq N^i}, \Qa^i_{\bullet\leq N^i}, \overline{W}^i_{\bullet,\bullet\leq N^i}$ and $\Ra^i_{\bullet,\bullet\leq N^i}$, with all the other choices being the same. But since we are applying $\tau_{\leq c}$, we may replace the $N^2$ truncations by $N^1$ truncations through a natural quasi-isomorphism by Lemma \ref{N-trunc}. Thus, it is reduced to case 1).

\textit{Functoriality:} 
Given $f:X\to X'$ of smooth schemes, we must choose as above in a compatible way. Firstly, we may choose compatible special hypercoverings $X_\bullet\to X$ and $X'_\bullet\to X'$ by Proposition \ref{sp hypercovering}.iii).

By Proposition \ref{simplicial compactification} (and its proof), we can always pick compatible compactifications $X\into \overline{X}$ and $X'\into \overline{X}'$ and find simplicial compactifications $\overline{X}_\bullet$ and $\overline{X}'_\bullet$ over $\overline{X}$ and $\overline{X}'$ fitting into a commutative diagram
$$\xymatrix{(X_\bullet,\overline{X}_\bullet)\ar[r]\ar[d] & (X'_\bullet,\overline{X}'_\bullet)\ar[d]\\
(X,\overline{X})\ar[r]& (X',\overline{X}').}$$

After choosing $\overline{U}'$ and $\mathcal{U}'$, we may choose a compatible affine covering of $\overline{X}$ to give $\overline{U}$, and a closed embedding into some $\mathcal{U}$ compatible. 

Fix $N$ as in (\ref{N,h}) for  some $h=c$ satisfying Theorem \ref{vanish rigid} for both $X$ and $X'$. After choosing any lifting $E'_N$ of $X'_N$ and an embedding into affine and projective space, we may pick lifts of $E_N$ and embeddings compatible as explained in the remark after introducing standard smooth algebras. This will give compatible morphisms
$$]X_{\bullet\leq N}[^\dagger_{\Fa_{\bullet\leq N}}\to ]X'_{\bullet\leq N}[^\dagger_{\Fa'_{\bullet\leq N}},$$
$$(X_{\bullet\leq N},Y_{\bullet\leq N},\Qa_{\bullet\leq N})\to (X'_{\bullet\leq N},Y'_{\bullet\leq N},\Qa'_{\bullet\leq N})$$
of $N$-truncated dagger spaces and rigid frames respectively.

By \cite[Proposition 11.5.2]{Chiarellotto2003}, we may also pick compatible refinements $V_{\bullet}$ and $V'_\bullet$ of 
$$(\cosk^X_N(\sk^X_N(X_\bullet))\times_X U, \cosk^{\overline{X}}_N(\sk^{\overline{X}}_N(\overline{X}_\bullet))\times_{\overline{X}}\overline{U})\text { and }$$
$$(\cosk^{X'}_N(\sk^{X'}_N(X'_\bullet))\times_{X'} U', \cosk^{\overline{X}'}_N(\sk^{\overline{X}'}_N(\overline{X}'_\bullet))\times_{\overline{X}'}\overline{U}')$$ over $(U,\overline{U})$ and $(U',\overline{U}')$ respectively.

Finally, when choosing $W'_{m,n}$ to be the closure of $V'_{m,n}$ in $\overline{V}'_{m,n}\times_k Y'_n$, from the diagram
$$\xymatrix{V_{m,n}\ar[d]\ar[rr] && \overline{V}_{m,n}\times_k Y_n\ar[d]\\
V'_{m,n}\ar[r] & W'_{m,n}\ar[r] & \overline{V}'_{m,n}\times_k Y'_n}$$
by the universal property of the schematic closure, we get a  closed immersion
$$W_{m,n}\into W'_{m,n}\times_{(\overline{V}'_{m,n}\times_k Y'_n)}(\overline{V}_{m,n}\times_k Y_n)$$
where $W_{m,n}$ is the closure of $V_{m,n}$ in $\overline{V}_{m,n}\times_k Y_n$. This gives compatible morphisms from the $X'$ version of (\ref{comparison maps}) to the $X$ version.
\end{proof}

\section{Application}

As an application, we consider the following $p$-adic \et ale motivic cohomology on smooth $k$-varieties (generalized in \cite[Appendix B]{Flach2016} to general $k$-varieties):
$$R\Gamma_c(X_{\et},\Z_p(n)):=R\text{lim}_\bullet R\Gamma_c(X_{\et},\\Z(n)/p^\bullet)$$
and
$$R\Gamma_c(X_{\et},\Q_p(n)):=R\Gamma_c(X_{\et},\Z_p(n))_\Q.$$
Here $\Z(n)$ is Suslin-Voevodsky's motivic complex defined in \cite[Definition 3.1]{Suslin2000} on the big category of smooth $k$-schemes $\Sm/k$. However, since we will be interested in a p-adic completion of this cohomology, we will use the identification
\begin{align}\label{motivic log}\Z/p^r(n)\cong W_r\Omega^n_{log}[-n]\end{align}
on $\Sm/k$ from \cite[Theorem 8.5]{Geisser2000b}, where $W_r\Omega^n_{log}$ (denoted $\nu^n_r$ there) is the subsheaf of $W_r\Omega^n$ \et ale locally generated by sections of the forms $\dlog f_1...\dlog f_n$, defined in \cite[II.5.7]{Illusie1979}.

As an application of Theorem \ref{comparison theorem}, we have the following result, which proves \cite[Conjecture 7.16.b]{Flach2016} for smooth schemes (using \cite[Theorem 4.3]{GeisserArithmeticcohomologyfinite2006a}):

\begin{thm}
For a separated, finite type smooth $k$-scheme $X$, and $n\in \Z$, there exists a quasi-isomorphism
$$R\Gamma(X_{\et},\Q_p(n))\stackrel{\sim}{\to}\left[R\Gamma_{rig,c}(X/K)^*\stackrel{\phi-p^{n-d}}{\longrightarrow}R\Gamma_{rig,c}(X/K)^*\right][-2d].$$
\end{thm}
Here, $\phi$ is the Frobenius on rigid cohomology, $R\Gamma_{rig,c}(X/K_0)^*:=R\text{Hom}(R\Gamma_{rig,c}(X/K_0),K)$ and $[A\to B]:=\text{Cone}(A\to B)[-1]$.
\begin{proof}
Using (\ref{motivic log}) we have that
$$\Rig(X,\Q_p(n))\cong  R\t{lim}R\Gamma(X_{et},W_\bullet\Omega^n_{X,\t{log}})_\Q[-n].$$
By \cite[I.Theorem 5.7.2.]{Illusie1979} we have a short exact sequence 
$$0\to W_\bullet\Omega^n_{X,\t{log}}\to W_\bullet\Omega^n_{X}\stackrel{1-F}{\to} W_\bullet\Omega^n_{X}\to 0$$
in $X_{et}$ , and by the proof of \cite[II.Proposition 2.1.]{Illusie1979},
$$W\Omega^n_X\cong R\t{lim}W_\bullet\Omega^n_X$$
which gives us
\begin{align}\label{rigid to cone}\Rig(X,\Q_p(n))\cong \left[R\Gamma(X_{et},W\Omega^n_{X})\stackrel{1-F}{\to}R\Gamma(X_{et},W\Omega^n_{X})\right]_\Q[-n].\end{align}

Next, by \cite[Corollary 2.4.12]{Ertl2013}, we have that all logarithmic Witt de-Rham sections are overconvergent, and that $1-F$ is still surjective when restricted to the overconvergent part; so we have a commutative diagram $X_{et}$ given by
$$\xymatrix{0\ar[r] & W\Omega^n_{X,\t{log}}\ar@{=}[d]\ar[r] &  W^\dagger\Omega^n_{X}\ar@{^{(}->}[d]\ar[r]^{1-F} &  W^\dagger\Omega^n_{X}\ar@{^{(}->}[d]\ar[r] & 0\\
0\ar[r] & W\Omega^n_{X,\t{log}}\ar[r] &  W\Omega^n_{X}\ar[r]^{1-F} &  W^\dagger\Omega^n_{X}\ar[r] & 0}$$
where the vertical arrows are given by inclusion, and both rows are short exact sequences. Thus, we get a natural quasi-isomorphism
\begin{align}\label{over to cone}\left[R\Gamma(X_{et},W\Omega^n_{X})\stackrel{1-F}{\to}R\Gamma(X_{et},W\Omega^n_{X})\right]_\Q[-n]\cong \left[R\Gamma(X_{et},W^\dagger\Omega^n_{X})\stackrel{1-F}{\to}R\Gamma(X_{et},W^\dagger\Omega^n_{X})\right]_\Q[-n].\end{align}

We consider the the Frobenius on $W^\dagger\Omega^\bullet_X$ by restricting that on $W\Omega^\bullet_X$. By the slope decomposition of $F$-crystals in \cite[Corollaire 5.3]{Illusie1979}
$$R\Gamma(X,W\Omega^\bullet_{X/k})\otimes_{W(k)}K\cong \bigoplus_{i=0} R\Gamma(X,W\Omega^i_{X/k})[-i]\otimes_{W(k)}K$$
we see the part with slope $p^n$ must be coming from $R\Gamma(X,W^\dagger\Omega^n)[-n]$, thus giving
\begin{align}\label{slope part}\left[R\Gamma(X_{et},W^\dagger\Omega^n_{X})\stackrel{1-F}{\to}R\Gamma(X_{et},W^\dagger\Omega^n_{X})\right]_\Q[-n]\cong \left[R\Gamma(X_{et},W^\dagger\Omega^\bullet_{X})\stackrel{p^n-\phi}{\to}R\Gamma(X_{et},W^\dagger\Omega^\bullet_{X})\right]_\Q.\end{align}

Then, by Theorem \ref{comparison theorem} we get that
$$\Rig(X/K)\stackrel{\sim}{\to} R\Gamma(X,W^\dagger\Omega^\bullet_{X/k})_\Q$$
so by (\ref{rigid to cone}), (\ref{over to cone}) and (\ref{slope part}) we have
\begin{align}\label{mot to rig}R\Gamma(X,\Q_p(n))\cong \left[\Rig(X/K) \stackrel{p^n-\phi}{\to} \Rig(X/K)\right].\end{align}

Finally, from \cite[Th\'{e}or\`{e}me 2.4]{Berthelot1997a} we can use Poincar\'{e} duality for rigid cohomology to get non-degenerate pairings
$$H^i_{\t{rig}}(X/K)\times H_{\t{rig,c}}^{2d-i}(X/K)\to H^{2d}_{\t{rig,c}}(X/K)\stackrel{\sim}{\to} K(-d)$$
compatible as F-crystals, where $K(-d)$ is viewed as $K$ with a Frobenius action given by multiplication by $p^d$. Thus, we have a natural quasi-isomorphism
$$\Rig(X/K)\stackrel{\sim}{\to}R\Gamma_{\t{rig,c}}(X/K)^*[-2d]:=\t{RHom}(\t{R}\Gamma_{\t{rig,c}}(X/K),K)[-2d]$$
and therefore,
\begin{align*}\Rig(X,\Q_p(n))\cong & \left[\Rig(X/K) \stackrel{p^n-\phi}{\to} \Rig(X/K)\right]\\
& \cong \left[\t{R}\Gamma_{\t{rig,c}}(X/K)^* \stackrel{p^{d-n}-\phi}{\to} \t{R}\Gamma_{\t{rig,c}}(X/K)^*\right][-2d].\end{align*}
\end{proof}

\bibliography{A_comparison_of_overconvergent_Witt_de-Rham_cohomology_and_rigid_cohomology_on_smooth_schemes}

\bibliographystyle{alpha}

\end{document}